\documentclass[11pt]{amsart}
\usepackage[pagebackref=true, colorlinks, allcolors=blue]{hyperref} 
\usepackage{breakurl}

\usepackage{amsfonts,amsmath,amsthm,amssymb}
\usepackage{tikz}
\usepackage{fancyhdr}
\usepackage{enumerate}
\usepackage{mathrsfs}  
\usepackage{cite}
\usepackage[alphabetic]{amsrefs}
\usepackage{colonequals}
\usepackage[boxed, section]{algorithm}
\usepackage{algpseudocode}
\usepackage[color = lime]{todonotes}
\usepackage{graphicx}
\usepackage[all]{xy}

\theoremstyle{plain}\newtheorem{thm}{Theorem}[section]
\theoremstyle{plain}\newtheorem{ithm}{Theorem}
\theoremstyle{definition}\newtheorem{defn}[thm]{Definition}
\theoremstyle{definition}
\theoremstyle{definition}
\theoremstyle{definition}
\theoremstyle{definition}
\theoremstyle{definition}
\theoremstyle{definition}
\theoremstyle{definition}
\theoremstyle{plain}\newtheorem{prop}[thm]{Proposition}
\theoremstyle{definition}\newtheorem{rmk}[thm]{Remark}
\theoremstyle{plain}\newtheorem{lem}[thm]{Lemma}
\theoremstyle{definition}
\theoremstyle{definition}
\theoremstyle{plain}\newtheorem{cor}[thm]{Corollary}
\theoremstyle{plain}\newtheorem{icor}[ithm]{Corollary}
\theoremstyle{plain}
\theoremstyle{definition}\newtheorem{irmk}[ithm]{Remark}
\numberwithin{equation}{section}

\renewcommand{\aa}{\mathbb{A}}
\newcommand{\cc}{\mathbb{C}}
\newcommand{\rr}{\mathbb{R}}
\newcommand{\pp}{\mathbb{P}}

\renewcommand{\qq}{\mathbb{Q}}
\newcommand{\zz}{\mathbb{Z}}
\newcommand{\ff}{\mathbb{F}}

\renewcommand{\gg}{\mathbb{G}}
\newcommand{\nn}{\mathbb{N}}

\newcommand{\A}{\mathcal{A}}
\newcommand{\B}{\mathcal{B}}

\newcommand{\I}{\mathcal{I}}

\renewcommand{\O}{\mathcal{O}}

\newcommand{\Q}{\mathcal{Q}}

\newcommand{\V}{\mathcal{V}}
\newcommand{\W}{\mathcal{W}}

\newcommand{\sB}{\mathscr{B}}

\newcommand{\sP}{\mathscr{P}}
\newcommand{\sQ}{\mathscr{Q}}

\newcommand{\comment}[1]{}

\newcommand{\Pic}{\operatorname{Pic}}

\newcommand{\Gal}{\operatorname{Gal}}

\newcommand{\Br}{\operatorname{Br}}
\newcommand{\Sym}{\operatorname{Sym}}

\newcommand{\disc}{\operatorname{disc}}

\newcommand{\Spec}{\operatorname{Spec}}

\newcommand{\Gr}{\operatorname{Gr}}
\newcommand{\Fr}{\operatorname{Fr}}
\newcommand{\GL}{\operatorname{GL}}
\newcommand{\inv}{\operatorname{inv}}

\newcommand{\CH}{\operatorname {CH}}

\newcommand{\defi}[1]{\textsf{#1}}

\title[A Brauer--Manin obstruction to weak approximation]{A transcendental Brauer--Manin obstruction to weak approximation on \\ a Calabi--Yau threefold}

\author[Sachi Hashimoto, Katrina Honigs, Alicia Lamarche, Isabel Vogt]{Sachi Hashimoto, Katrina Honigs, Alicia Lamarche, Isabel Vogt  (with an appendix by Nicolas Addington)}

\newcommand{\ContactInfo}{{
  \bigskip
  \footnotesize
  
  Nicolas Addington, \textsc{Department of Mathematics, University of Oregon, Eugene, OR 97403}\par\nopagebreak
  \textit{E-mail address}: \texttt{adding@uoregon.edu}\par\nopagebreak
 \textit{URL}: \url{http://pages.uoregon.edu/adding/}

  \bigskip
  
  Sachi Hashimoto, \textsc{Department of Mathematics and Statistics, Boston University, Boston, MA 02215}\par\nopagebreak
  \textit{E-mail address}: \texttt{svh@bu.edu}\par\nopagebreak
  \textit{URL}: \url{http://math.bu.edu/people/svh/}

  \bigskip

  Katrina Honigs, \textsc{Department of Mathematics, University of Oregon, Eugene, OR 97403}\par\nopagebreak
  \textit{E-mail address}: \texttt{honigs@uoregon.edu}\par\nopagebreak
 \textit{URL}: \url{http://pages.uoregon.edu/honigs/}

  \bigskip
    
  Alicia Lamarche, \textsc{Department of Mathematics, University of Utah, Salt Lake City, UT 84112}\par\nopagebreak
  \textit{E-mail address}: \texttt{lamarche@math.utah.edu}\par\nopagebreak
 \textit{URL}: \url{http://alicia.lamarche.xyz}

  \bigskip
  Isabel Vogt, \textsc{Department of Mathematics, University of Washington, Seattle, WA 98195}\par\nopagebreak
  \textit{E-mail address}: \texttt{ivogt.math@gmail.com}\par\nopagebreak
 \textit{URL}: \url{http://faculty.washington.edu/ivogt/}

}}

\date{\today}

\begin{document}

\begin{abstract}
In this paper we investigate the $\qq$-rational points of a class of simply connected Calabi--Yau threefolds, which were originally studied by Hosono and Takagi in the context of mirror symmetry. These varieties are defined as a linear section of a double quintic symmetroid;  their points correspond to rulings on quadric hypersurfaces.
They come equipped with a natural $2$-torsion Brauer class. Our main result shows that under certain conditions, this Brauer class gives rise to a transcendental Brauer--Manin obstruction to weak approximation.

Hosono and Takagi showed that over $\cc$ each of these Calabi--Yau threefolds $Y$ is derived equivalent to a Reye congruence Calabi--Yau threefold $X$. We show that these derived equivalences may also be constructed over $\qq$, and we give sufficient conditions for $X$ to not satisfy weak approximation. In the appendix, N.~Addington exhibits the Brauer groups of each class of Calabi--Yau variety over $\cc$.
\end{abstract}

\maketitle

\section{Introduction}

Let $Y$ be a smooth projective variety defined over a number field $K$. 
The $K$-rational points of $Y$ embed diagonally in the adelic points  $Y(K) \subseteq Y(\aa_K) = \prod_{v} Y(K_v)$.  
Suppose $Y(\aa_K) \neq \emptyset$, that is, that $Y(K_v)$ is nonempty for all places $v$ of $K$. We say that $Y$ satisfies \defi{weak approximation} if $\overline{Y(K)} = Y(\aa_K)$  in the adelic topology. Let $\Br(Y)$ denote the Brauer group of $Y$.
Given any element $\alpha \in \Br(Y)$, there is an intermediate set $Y(\aa_K)^\alpha$ that contains the closure of the $K$-rational points
\[\overline{Y(K)} \subset Y(\aa_K)^\alpha \subset Y(\aa_K). \]
If $Y(\aa_K)^\alpha \neq Y(\aa_K)$, we say that $\alpha$ gives a \defi{Brauer--Manin obstruction to weak approximation} on $Y$.
Such an obstruction is \defi{transcendental} if $\alpha_{Y_{\bar{K}}}$ is nontrivial; otherwise it is algebraic.  

In this paper, we construct examples of Calabi--Yau threefolds $Y/\qq$ that naturally come with a $2$-torsion class $\alpha \in \Br(Y)[2]$ providing a transcendental Brauer--Manin obstruction to weak approximation.

Transcendental Brauer classes are often more difficult to describe explicitly than their algebraic counterparts. 
While the first examples of transcendental Brauer--Manin obstructions were constructed using extra structure such as fibrations \cite{Harari, Wittenberg, Ieronymou}, V\'arilly-Alvarado along with collaborators Hassett--Varilly \cite{HVAV}, Hassett \cite{HVA},  McKinnie--Sawon--Tanimoto \cite{MSTVA}, and Berg \cite{BVA} have demonstrated that a lattice-theoretic framework can be leveraged to exhibit transcendental Brauer--Manin obstructions on general K3 surfaces.  

Our Calabi--Yau threefold examples are of a similar flavor to those in \cite{HVAV}; in particular, we also leverage the data of linear spaces in a family of quadrics and prove that evaluation of our Brauer class has a very elegant and satisfying geometric description detailed in Sections~\ref{sec:Yalpha} and~\ref{sec:obstr_weak}.  Relatively few examples of transcendental obstructions on threefolds are known, and to the best of our knowledge, these are the first such examples on Calabi--Yau threefolds.
The Calabi--Yau threefolds we consider are also of independent interest in string theory. They have been studied over $\cc$ in a series of papers of Hosono and Takagi \cites{ht2,ht3,ht1},
but the present paper initiates the study of their arithmetic.

The varieties $Y = Y_P$ we consider are double covers of the locus of singular quadrics in a $4$-dimensional linear system $P$ of quadric hypersurfaces in $\pp^4$ over $\qq$, as we now describe.
A singular quadric threefold in $\pp^4$ is the cone over a quadric surface of the same rank in $\pp^3$.
Since rank $4$ quadric surfaces are doubly ruled by lines, a rank $4$ quadric hypersurface in $\pp^4$ is doubly ruled by $2$\nobreakdash-planes; furthermore, the $2$\nobreakdash-planes in a given ruling are parametrized by a conic. Similarly, a rank $3$ quadric in $\pp^4$ has a unique ruling by $2$\nobreakdash-planes, and the $2$-planes in this ruling are again parametrized by a conic.
Given a suitably generic (we call the precise condition \defi{regular}, see Definition~\ref{def:reg}) $4$\nobreakdash-dimensional linear system $P$ of quadric hypersurfaces in $\pp^4$,
we may associate to it a variety $Y_P$ that parameterizes the singular quadrics (of ranks $3$ and $4$) in $P$ along with a ruling of $2$-planes on the quadric.  
The incidence variety parameterizing singular quadrics in $P$ together with a choice of $2$-plane is then a conic bundle over $Y_P$.
This  \'etale $\pp^1$-bundle over $Y_P$ gives rise to a nontrivial class $\alpha = \alpha_P$ in $\Br(Y_P)$ \cite{ht1}*{Proposition~3.2.1}.  We show that $\alpha_P$ can provide a transcendental Brauer--Manin obstruction to weak approximation on $Y_P$.
 
It is not a priori clear whether the rational points of such $Y$ should satisfy weak approximation: some simple obstructions to weak approximation vanish.  The threefolds $Y$ that arise in this way are simply connected \cite{ht2}*{Proposition~4.3.4}, so the fundamental group does not provide an obstruction to weak approximation (see \cite{Minchev}*{Theorem~2.4.4} for a description of this obstruction).
Furthermore, there is no algebraic Brauer--Manin obstruction to weak approximation: Proposition \ref{prop:BrY} shows that the algebraic part of the Brauer group of $Y$ is trivial. In fact, it is a consequence of Proposition \ref{prop:brauerX} by
Addington that $\Br(Y_{\bar{\qq}})$ is generated by~$\alpha_{\bar{\qq}}$.

Our main result shows that with  fairly mild restrictions on the choice of quadrics defining the linear system \(P\), the resulting Brauer class $\alpha_P$ on the threefold $Y_P$ always provides a transcendental obstruction to weak approximation.

\begin{ithm}\label{thm:main}
Suppose that $P$ is a regular $4$-dimensional linear system of quadrics defined over $\qq$ that contains a quadric with signature
$(+1, +1, +1, +1, 0)$ or $(-1, -1, -1, -1, 0)$ over $\rr$.  
If $Y_P(\aa_\qq) \neq \emptyset$, then we have the proper containment
\[\emptyset \neq Y_P(\aa_{\qq})^{\alpha_P} \subsetneq Y_P(\aa_{\qq}).\]
In particular, the smooth Calabi--Yau threefold $Y_P$ does not satisfy weak approximation.
\end{ithm}
\goodbreak

The proper containment $Y_P(\aa_{\qq})^{\alpha_P} \subsetneq Y_P(\aa_{\qq})$ arises in these cases because the Brauer class $\alpha_P$ has nonconstant evaluation at the real points of $Y_P$: 
we have chosen $P$ so that $Y_P(\rr)$ contains a point where $\alpha_P$ has nontrivial evaluation and we show that for \emph{any} choice of $P$, the set $Y_P(\rr)$ always contains a point where $\alpha_P$ has trivial evaluation.

Linear systems $P$ containing a quadric of the desired signature and giving rise to locally soluble \(Y_P\) are common; in the proof of the following corollary, we produce one such example.

\begin{icor}\label{cor:easy}
  Let $P$ be a general $4$-dimensional linear system of quadrics defined over $\qq$ that contains the quadric $x_0^2 + x_1^2 + x_2^2 + x_3^2 = 0$.  Then $Y_P(\aa_{\qq}) \neq \emptyset$ and \(Y_P\) does not satisfy weak approximation.
\end{icor}

The proof of Theorem \ref{thm:main} centers upon controlling the evaluation of \(\alpha_P\) at real points of \(Y_P\).
For other places, we show that for a random choice of \(P\), the class \(\alpha_P\) has trivial evaluation at every \(p\)-adic point of \(Y_P\) for \emph{all} finite places \(p\) with probability at least \(.73\).   
To make this precise, we must explain how we generate a random \(P\).  We do this by specifying (integral) basis vectors for the \(5\)-dimensional subspace of \(\aa^{15}\) corresponding to this plane. We make this density statement precise in Section \ref{expectedbehavior}. 

\begin{ithm}\label{thm:example}
Among all $4$-dimensional linear systems of quadrics $P$ defined over $\zz$, those where the Brauer class $\alpha_P$ has trivial evaluation at all of the points in $Y_P(\qq_p)$ for every prime $p$ occur with density at least $.73$.

Furthermore, the following five quadrics in $\pp^4$:
\begin{align*}
&x_0x_1 + x_2x_3\\
&x_0^2 + x_0x_1 + 3x_0x_2 + 2x_1^2 + x_1x_2 + x_1x_3 + 5x_2^2 + x_2x_3 + x_3^2\\
&x_0^2 + 2x_0x_2 + x_0x_3 + x_0x_4 + x_1x_2 + x_1x_4 + x_2x_4 + 4x_3x_4\\
&4x_0^2 + 4x_1^2 + 3x_1x_3 + 2x_1x_4 + 3x_2^2 + x_2x_4 + x_4^2\\
&x_0^2 + x_0x_2 + x_0x_3 + 4x_0x_4 + 3x_1x_2 + 4x_1x_4 + 2x_2x_3 + 2x_3^2 + x_4^2
\end{align*}
generate a regular linear system $P$ for which we have the containments
\[\emptyset \neq Y_P(\qq) \subset Y_P(\aa_{\qq})^{\alpha_P} \subsetneq Y_P(\aa_{\qq}),\]
and $\alpha_P$ has trivial evaluation at every point of $Y_P(\qq_p)$ for every prime $p$.
\end{ithm}

We also give an explicit quaternion algebra over $\qq(Y_P)$  in \eqref{explicit_alpha_brauer} that represents $\alpha_P$ in this case.

\begin{irmk}
In Theorems \ref{thm:main} and \ref{thm:example}, the fact that $\alpha_P$ \emph{always} has trivial evaluation at some real point of $Y_P$ allowed us to specify information only about the real place (i.e., the signature) for one quadric in $P$.  Instead, we can use a finite place $p$ to obstruct weak approximation if we specify $p$-adic data for \emph{two} quadrics in $P$.  For example, working with the prime $p=3$, we show in Proposition \ref{prop:q3} that for
$P$ a general $4$-dimensional linear system of quadrics over $\qq$ containing
$6x_0^2 -3x_1^2 + 2x_2^2 - x_3^2=0 $ and
$x_0x_1 + x_2x_3= 0$, the corresponding $Y_P$ does not satisfy weak approximation.
\end{irmk}

Our examples are also of interest in the study of derived equivalence, particularly in comparing the rational points of derived equivalent varieties. For instance, the question posed by Hassett and Tschinkel \cite{HT} 
of how the existence and density of $\qq$-rational points on derived equivalent K3 surfaces over $\qq$ compares remains open.
It was shown in \cite{aafh} that for hyperk\"ahler fourfolds and abelian varieties of dimension at least $2$, the existence of a $\qq$-rational point is not a derived invariant.
Hence Calabi--Yau threefolds are a very interesting intermediate type of variety to study with regard to this question.
Hosono and Takagi showed that each of the Calabi--Yau threefolds $Y_P$ we study here is derived equivalent, but not birational, to another Calabi--Yau threefold $X_P$ also attached to the linear system \(P\).
We prove that these derived equivalences may be constructed over $\qq$ as well, and show that in many cases the $\qq$-points of these derived equivalent varieties behave similarly. We believe that this class of examples will open many questions for future study.

\smallskip

In Section \ref{sec:Yalpha} we outline the construction of the Calabi--Yau threefolds $Y_P$ along with the Brauer class $\alpha_P \in \Br(Y_P)[2]$ and give some preliminary results about these objects.
In Lemma~\ref{lem:evaluation_real} we show that for all  $(Y_P, \alpha_P)$,  the set $Y_P(\rr)$ contains a point at which $\alpha_P$ has trivial evaluation.
In \eqref{explicit_alpha_brauer} we give explicit equations for the quaternion algebra in Theorem~\ref{thm:example}.
In Section \ref{sec:obstr_weak} we review the construction of the intermediate set $Y_P(\aa_{\qq})^{\alpha_P}$, discuss the evaluation of $\alpha_P$, and give equivalent geometric criteria for this evaluation to be trivial.  We prove Theorem~\ref{thm:main} and Corollary \ref{cor:easy} in Section \ref{sec:mainproof}. We prove that a random $P$ gives an $\alpha_P$ with constant trivial evaluation at all finite primes in Section \ref{expectedbehavior} and discuss the explicit example of Theorem \ref{thm:example} in Section \ref{sec:example}. Finally, we end with a discussion of the derived equivalent threefolds $X_P$ in Section~\ref{sec:derivedequiv}. We construct $X_P$ over $\qq$ and show that kernel of Hosono and Takagi's equivalence is defined over $\qq$. We show in Lemma~\ref{lem:waX} that if the linear system $P$ contains a singular quadric defined over $\qq$, then $X_P$ does not satisfy weak approximation.

\subsection{Acknowledgments}

We thank Nick Addington for many helpful conversations. In particular, we thank him for the idea behind Lemma~\ref{lem:evaluation_real} and for providing the appendix.  We also thank Brendan Hassett, Daniel Krashen, Bjorn Poonen, Anthony V\'arilly-Alvarado, and Bianca Viray for helpful discussions and comments on an earlier draft.  We thank the anonymous referees for many helpful comments and suggestions that have improved the results and exposition of the paper.  This collaboration started at the American Mathematical Society Mathematics Research Communities on ``Explicit methods in arithmetic geometry in characteristic $p$''; we thank the American Mathematical Society and National Science Foundation for funding this workshop.

\subsection{Funding}

S.H. was supported by a Clare Boothe Luce Fellowship (Henry Luce Foundation) and by National Science Foundation grant DGE-1840990. A.L. was supported by National Science Foundation grants RTG-1840190 and DMS-2103271. I.V. was supported by National Science Foundation grant DMS-1902743.

\subsection{Code Availability}

The code associated to this paper can be found at the github repository \cite{githubrepo} and on the supplementary files with the arXiv preprint of this paper.

\section{The Calabi--Yau threefold $Y$ and Brauer class $\alpha$}\label{sec:Yalpha}

We begin by reviewing the construction of the Calabi--Yau threefold $Y_P$ associated to a regular $4$-dimensional linear system \(P\) of quadrics in $\pp^4$.  This construction proceeds through the auxiliary incidence variety $Z_P$, which gives the nontrivial $2$-torsion Brauer class $\alpha_P \in \Br(Y)[2]$.  We then give a quaternion algebra over the function field of $Y_P$ that represents it.

When the linear system \(P\) is unambiguous, we write \(Y = Y_P\), \(Z = Z_P\) and \(\alpha = \alpha_P\) for brevity.

\subsection{The variety $Y$}\label{sec:Y}

We now give a detailed description of the Calabi--Yau threefolds $Y_P$ on which we will obstruct weak approximation. These varieties are constructed from the geometry of linear spaces in a family of quadrics; see \cite{Reid-thesis}*{Chapter 1} or \cite{Harris}*{Lecture~22} for an introduction to this topic.  

Let $K$ be a field not of characteristic $2$. 
Recall that $Y_P$ is attached to a choice of linear system $P$.  Explicitly, let $Q_0, \dots, Q_4$ be five independent quadrics in $\pp_{K}^4$.  These define a linear system $P \simeq \pp^4_{(t_0, \dots, t_4)}$ by
\[ P \colonequals \left| t_0 Q_0 + t_1 Q_1 + t_2 Q_2 + t_3 Q_3 + t_4 Q_4 \right| \subseteq \pp H^0(\pp^4, \O_{\pp^4}(2)) \simeq \pp_{K}^{14}. \]

Equivalently, by the correspondence between symmetric bilinear forms  and quadratic forms, we may view $P$ as defining a linear system of $(1,1)$-divisors in $\pp^4 \times \pp^4$.  We will always impose the following genericity assumption on the linear system.

\begin{defn}\label{def:reg}
A linear system $P\simeq \pp_K^4$ of $(1,1)$-divisors in $\pp^4 \times \pp^4$ is \defi{regular} if their intersection is smooth and does not intersect the diagonal in~$\pp^4 \times \pp^4$. 
\end{defn}

\begin{rmk}We adopt the above perspective of viewing $P$ as a linear system of $(1,1)$-divisors in $\pp^4\times \pp^4$ in anticipation of the construction of the Calabi-Yau threefold $X$ in Section \ref{sec:defx}. Note that there are two distinct copies of $\pp^4$, with one being dual to the other. This projective duality arises in the study of the derived equivalence between $X$ and $Y$. 
\end{rmk}

By \cite{ht3}*{Proposition~2.1}, this may be equivalently phrased purely in terms of the incidence geometry of lines in the quadrics: a linear system is regular if and only if
it is base-point free, and 
any (geometric) line $\ell \subseteq \operatorname{Sing}(Q)$ for some $Q$ in $P(\overline{K})$ is not contained in a sublinear system of dimension at least $2$ in $P$.
In particular, regularity guarantees that the singular quadrics in $P$ have rank $3$ or $4$.

Suppose that $P$ is a regular linear system.  Write $\Q \xrightarrow{f} P$ for the universal quadric over $P$.  Over a point $t = (t_0 : t_1 : t_2: t_3 : t_4) \in P$, we write $Q_t$ for the corresponding quadric in $\pp^4$, with equation
$\sum_{i=0}^4 t_i Q_i = 0$.
 Let $Z_P$ be the relative Fano scheme of $2$-planes contained in the fibers of $f$
\[Z_P \colonequals F(2,\Q/P) =  \Bigl\{ (t, [\Lambda] ) \in P \times \gg(2, \pp^4) : \Lambda \subset Q_t, \text{i.e., } \sum_i t_i Q_i|_{\Lambda} = 0\Bigr\}. \]
Then $Z_P$ admits two natural projections 
\begin{equation}\label{projections} Z_P \xrightarrow{\pi_1} P, \qquad Z_P \xrightarrow{\pi_2} \gg(2, \pp^4).\end{equation}
Since any $3$-dimensional quadric hypersurface containing a $2$-dimensional linear space is necessarily singular, the image of the first projection of $Z_P$ to $P$ is the locus $H \subset P$ of singular quadrics, that is, cones over quadric surfaces in $\pp^3$.  As we observed in the introduction, lines on quadric surfaces come in two rulings, so the map $Z _P\to H$ does not have connected fibers.
The threefold $Y_P$ of interest can be seen in the nontrivial Stein factorization:
\[\pi_1 \colon Z_P \xrightarrow{\text{conic bundle}} Y_P \xrightarrow{2:1}  H \subseteq P.\]
Hosono and Takagi show that if $P$ is regular, then $Y_P$ is smooth \cite{ht3}*{Proposition~3.11}. 
The conic bundle $Z_P$ gives us our $2$-torsion Brauer class $\alpha_P \in \Br(Y)[2]$, and $(\alpha_P)_{\bar{K}}$ is nontrivial by \cite{ht1}*{Proposition~3.2.1}.

\subsection{Rational points on $Y$}
For any field $L/K$, an $L$-rational point on $Y_P$ corresponds to a singular quadric from $P$ along with a choice of ruling both defined over $L$.  Given a rank $4$ (singular) quadric $Q$ in $\pp^4(L)$, 
we write $\bar{Q}$ for the smooth quadric in $\pp^3$ over which $Q$ is a cone.
By regularity of $P$, the points in $Y_P$ are over quadrics in
$\pp^4(L)$
of ranks only $3$ or $4$, hence the finite morphism $Y\to H$ is ramified over the locus of quadrics of rank~$3$.
The rulings of $2$-planes on~$Q$ (equivalently, lines on $\bar{Q}$) are defined over $L\bigl(\sqrt{\disc(\bar{Q})}\bigr)$ \cite{wang}*{Section~2.2, footnote on p.~366}.  Therefore, an $L$-point in the (open) rank $4$ locus of $H$ lifts to an $L$-point of $Y_P$ if and only if the corresponding quadric in $\pp^3$ has square discriminant.
The one ruling of $2$-planes on a quadric of rank $3$ is defined over $L$, and thus any $L$-point in the rank $3$ locus of $H$ lifts to an $L$-point of $Y_P$.

Over the generic point of $H$, the corresponding quadric $\Q_{\qq(H)}$ has rank exactly $4$; as such it is the cone over a smooth rank $4$ quadric $\bar{\Q}_{\qq(H)}$ in $\pp^3$.  Therefore the function field $\qq(Y)$ is the discriminant quadratic extension 
\begin{equation}\label{eq_disc}
\qq(Y) = \qq(H)\Bigl( \sqrt{\disc(\bar{\Q}_{\qq(H)})}\,\Bigr)
\end{equation}
of $\qq(H)$ over which the rulings of $\bar{\Q}_{\qq(H)}$ are defined.
We make this explicit in Lemma~\ref{lem:explicit_fun_field}.
Moreover, $Y$ is determined by its function field: since $Y$ is a Stein factorization it is normal and therefore by \cite[I.1 \S10 Proposition~10.1]{sga}
it is the normalization of $H$ inside the field extension $\qq(Y)$.  

\begin{defn}
  Given a quadratic form $Q$, we call $B_Q(x,y) = Q(x+y) -Q(x) - Q(y)$ the associated symmetric bilinear form.
  Write $\B$ for the universal symmetric bilinear form corresponding to the universal quadric $\Q$.
\end{defn}

Note that $2Q(x) = B_Q(x,x)$, and hence if $\operatorname{char}(K)\neq 2$, the quadratic form \(Q(x)\) can be recovered from the bilinear form \(B_Q(x,x)\). 

Given a symmetric $5\times5$ matrix $B$
and subsets $I, J \subseteq \{0,\dots, 4\}$, we define
$B_{I,J}$ to be the minor of $B$ obtained by deleting the $i$th rows and $j$th columns for $i \in I, j \in J$ (i.e. $B_{I,J}$ is the the determinant of the resulting submatrix).

\begin{lem}\label{lem:explicit_fun_field}
For any index $i$, if $\B_{i,i}$ is not identically $0$ on $H$, then
\[\qq(Y) = \qq(H)\left(\sqrt{\left(\B_{\qq(H)}\right)_{i,i}}\,\right).\]
\end{lem}
\begin{proof}
Fix an $i$ such that $0\leq i\leq 4$.
On the open locus in $H$ where $\B_{i,i} \neq 0$, the matrix $\B$ has rank exactly $4$, and therefore has a $1$-dimensional kernel $V$ not contained in the hyperplane $x_i =0$.
Let $W$ be any codimension $1$ subspace that does not contain $V$. The restriction $\B|_W$ is therefore a nondegenerate bilinear form.  Up to isomorphism, this bilinear form is independent of the choice of $W$.  Therefore, for any $W$, the discriminant of $\B|_W$ is equal (modulo squares) to $\disc(\bar{\Q}_{\qq(H)})$.  The result follows by applying this to $W = (x_i =0)$.
\end{proof}

Geometrically, this argument shows that if $\B_{i,i}$ and $\B_{j,j}$ are nonzero, then the intersections of $\Q_{\qq(H)}$ with the hyperplanes $x_i = 0$ and $x_j = 0$, respectively, are isomorphic
smooth rank $4$ quadrics in $\pp^3$ over which $\Q_{\qq(H)}$ is a cone.

\medskip

The other projection $\pi_2: Z_P \to \gg(2, \pp^4)$ is also illuminating when studying points on $Y_P$.
Containing a fixed $2$-plane $\Lambda_0 \subset \pp^4$ is a linear condition on the quadrics in $P$, so the fibers of $\pi_2$ over each point $[\Lambda_0] \in \gg(2, \pp^4)$ are projective spaces.  Let $\Sigma_P \subset \gg(2, \pp^4)$ denote the image~$\pi_2(Z_P)$ and let $\mathcal{S}$ be the tautological subbundle on $\gg(2, \pp^4)$.
Then the subvariety $\Sigma_P$ can be viewed as a degeneracy locus, as we now explain.  The quadrics $Q_i$ give rise to sections $s_i$ of the rank $6$ vector bundle $\Sym^2 \mathcal{S}^\vee$ on $\gg(2, \pp^4)$ by restriction of $Q_i$ to each $[\Lambda] \in \gg(2, \pp^4)$.  If $[\Lambda] \in \gg(2, \pp^4)$ is in $\Sigma_P$, then there exists some point $(t_0: \cdots : t_4) \in P$ such that
\[\sum_{i=0}^4 t_i Q_i|_{\Lambda} = 0.\]
Therefore, we may equivalently define $\Sigma_P$ as the locus in $\gg(2, \pp^4)$ where $s_0, \dots, s_4$ fail to be linearly independent, and hence it is a degeneracy locus.  By \cite{3264}*{Theorem~5.3(b)}, the class of $\Sigma_P$ in $\CH^*(\gg(2, \pp^4))$ is $c_2( \Sym^2 \mathcal{S}^\vee)$. 

\begin{lem}\label{lem:evaluation_real}
Let $P$ be any $4$-plane in $\pp^{14}_\qq$ and let \(Y_P\) and \(Z_P\) be as defined in Section \ref{sec:Y}.  Then $Z_P(\rr)\neq \emptyset$; in particular, $Y_P(\rr) \neq \emptyset$.
\end{lem}
\begin{proof}
We prove this lemma using topological arguments involving characteristic classes of real vector bundles; see \cite{hatcher}*{Section 3.3} for more details.
  Given a choice of five quadrics $Q_0, \dots, Q_4$ (defined over $\qq$) that span $P$, we may consider the construction described above over $\rr$.
Analogous to the second Chern class, the class of $(\Sigma_P)_\rr\subset\gg(2, \pp^4_{\rr})=:\gg_\rr$ in the cohomology $H^*(\gg_\rr, \zz/2\zz)$ is given by the second
Stiefel-Whitney class $w_2 (\Sym^2 \mathcal{S}_{\rr}^\vee)$ \cite{hatcher}*{Proposition 3.21}. 

A standard computation using the splitting principle shows 
$$w_2(\Sym^2 \mathcal{S}^\vee_\rr)  = w_1(\mathcal{S}^\vee_\rr)^2 + w_2(\mathcal{S}^\vee_\rr).$$
Finally, we observe that $H^2(\gg_\rr, \zz/2\zz)$ is freely generated over $\zz/2\zz$ by $w_1(\mathcal{S}^\vee_\rr)^2$ and $w_2(\mathcal{S}^\vee_\rr)$
\cite{milnorstasheff}*{Problem~6-B, Theorem~7.1}. 
Therefore, $[(\Sigma_P)_\rr]$ is a nontrivial class in $H^2(\gg_\rr, \zz/2\zz)$, and hence $(\Sigma_P)_{\rr}(\rr)$ cannot be empty.  Since the fibers of $\pi_2$ over each real point of $(\Sigma_P)_{\rr}$ are projective spaces, this implies that $Y(\rr)$ is nonempty as well.
\end{proof}

\subsection{\boldmath The Brauer class $\alpha$}
In this section, we begin by explicitly describing a quaternion algebra over the function field of $Y = Y_P$ that is the image of $\alpha = \alpha_P$ in $\Br(\qq(Y_P))$.  Recall that the conic bundle $Z = Z_P$ over $Y$ provides the nontrivial Brauer class $\alpha \in \Br(Y)$. We continue to fix a field $K$ not of characteristic $2$. 

\begin{lem}\label{lem:lines}
Let $Q \subseteq \pp_K^3$ be a smooth quadric with square discriminant (i.e., whose rulings are defined over $K$). Let $[\Gamma]$ in $(\pp^3)^\vee(K)$ be a hyperplane transverse to $Q$. Then the variety of lines on $Q$ in a given ruling is isomorphic to the conic $\Gamma \cap Q$. 
\end{lem}
\begin{proof}
Each line on $Q$ intersects $\Gamma \cap Q$ in a unique point.  Conversely,
the tangent hyperplane section of $Q$ at any point of $\Gamma \cap Q$ is the union of a unique line from each ruling.
\end{proof}

In particular, applying Lemma \ref{lem:lines} to the generic quadric, we have that the generic fiber $Z|_{\qq(Y)}$ of the conic bundle defining $\alpha$ is isomorphic to a transverse hyperplane section of $\bar{\Q}_{\qq(Y)}$.
We now make this more explicit in the generic situation, which will include our example in Theorem \ref{thm:example}. Recall that we write $\B = \B_P$ for the universal symmetric bilinear form over $P$. 

\begin{lem}\label{lem:explicit_Br_class}
Suppose that the principal minors $\B_{1234,1234}$, $\B_{234,234}$, $\B_{34, 34}$, and $\B_{4,4}$ are not identically zero on $H$.  Then the quaternion algebra
\[\left( \frac{\B_{234,234}}{(\B_{1234,1234})^2}, \frac{\B_{34,34}}{\B_{234,234}\B_{1234,1234}} \right) \in \Br\left(\qq(Y)\right)[2] \]
represents the restriction of $\alpha$.
\end{lem}
\begin{proof}
To show this result, we will use our hypotheses on the minors of $\B$ to
make a nice choice of basis with which to diagonalize $\Q_{\qq(H)}$, yielding the quaternion algebra above.
  
Recall that $\B$ has rank 4 over $\qq(H)$.
  Because of our hypotheses on the minors of $\B$, the matrix in the linear system below is invertible, allowing us to use Cramer's rule to solve the system. The vector $v$ satisfying the linear system is in the kernel of the generic bilinear form $\B|_{\qq(H)}$.
Applying Cramer's rule gives us the entries of $v$ shown on the right.
  \[ \begin{pmatrix} &&&&&0 \\ &&&&&0 \\ &&\B &&&0 \\ &&&&&0 \\&&&&&1 \\ 0 & 0& 0&0&1&0  \end{pmatrix} \begin{pmatrix} \\ \\ v \\ \\ \\ \ast \end{pmatrix} = \begin{pmatrix} 0\\0 \\ 0 \\ 0 \\ \ast \\ \B_{4,4} \end{pmatrix},
    \quad\quad 
v = \begin{pmatrix} \B_{4,0}\\ - \B_{4,1}\\ \B_{4,2}\\ -\B_{4,3}\\ \B_{4,4} \end{pmatrix}.
\]
By hypothesis, the last entry of 
$v$ is nonzero, and so
$\B_{\qq(H)}$ is nondegenerate when restricted to the subspace generated by $e_0, e_1, e_2, e_3$.  We may therefore carry out Gram--Schmidt orthonormalization on this subspace to diagonalize $\Q_{\qq(H)}$.  
Since Gram--Schmidt orthonormalization only involves adding multiples of the previous columns, the diagonalized matrix must have the same leading principal minors: 
$\B_{1234,1234},\B_{234,234}/\B_{1234,1234}, \B_{34,34}/\B_{234,234}, \B_{4,4}/\B_{34,34},0$.

Over $\qq(Y)$, the conic $Z|_{\qq(Y)}$ is isomorphic to
\[ \left(\B_{1234,1234}\right)x_0^2 + \left( \frac{\B_{234,234}}{\B_{1234,1234}} \right) x_1^2 + \left( \frac{\B_{34,34}}{\B_{234,234}} \right) x_2^2 = 0,\]
which corresponds (see for instance \cite{KKS}*{Section~8.2}) to the
desired quaternion algebra.
\end{proof}

For the choice of five quadrics in Theorem \ref{thm:example}, we have that $\alpha$ is therefore represented by the quaternion algebra $(\frac{a_1}{a_2}, \frac{b_1}{b_2})$, where
\begin{equation}\label{explicit_alpha_brauer}
\begin{split}
a_1 = & -t_0^2 - 2t_0t_1 + 7t_1^2 + 8t_1t_2 + 48t_1t_3 + 16t_2t_3 + 64t_3^2 + 8t_1t_4 + 16t_3t_4, \\
a_2 = & 4t_1^2 + 8t_1t_2 + 4t_2^2 + 32t_1t_3 + 32t_2t_3 + 64t_3^2 + 8t_1t_4 + 8t_2t_4\\
&+ 32t_3t_4 + 4t_4^2, \\
b_1 =&  -10t_0^2t_1 - 14t_0t_1^2 + 38t_1^3 + 10t_0t_1t_2 + 36t_1^2t_2 + 4t_0t_2^2 - 18t_1t_2^2 - 6t_0^2t_3\\
&- 12t_0t_1t_3 + 442t_1^2t_3  
+ 96t_1t_2t_3 - 40t_2^2t_3 + 928t_1t_3^2 + 96t_2t_3^2 + 384t_3^3\\
&+ 20t_0t_1t_4 + 62t_1^2t_4 + 14t_0t_2t_4 - 30t_1t_2t_4 - 14t_2^2t_4
+ 112t_1t_3t_4\\
&- 80t_2t_3t_4 + 96t_3^2t_4 + 6t_0t_4^2 - 28t_1t_4^2 - 30t_2t_4^2 - 80t_3t_4^2 - 18t_4^3, \\
b_2 =& -2t_0^2t_1 - 4t_0t_1^2 + 14t_1^3 - 2t_0^2t_2 - 4t_0t_1t_2 + 30t_1^2t_2 + 16t_1t_2^2 - 8t_0^2t_3\\
&- 16t_0t_1t_3 + 152t_1^2t_3 + 192t_1t_2t_3 
+ 32t_2^2t_3 + 512t_1t_3^2 + 256t_2t_3^2 + 512t_3^3\\
&- 2t_0^2t_4 - 4t_0t_1t_4 + 30t_1^2t_4 + 32t_1t_2t_4 + 192t_1t_3t_4 + 64t_2t_3t_4 \\
& + 256t_3^2t_4 + 16t_1t_4^2 + 32t_3t_4^2 .
\end{split}
\end{equation}
\goodbreak

We conclude this section with the computation of the Brauer group of $Y$.  

\begin{prop}\label{prop:BrY}
Assume that $P$ is regular and write \(Y = Y_P\).  Then we have that 
\begin{enumerate}[{\rm (i)}]
\item $\Pic(Y_{\bar{\qq}}) \simeq \zz$.
\item $\Br_1(Y)/\Br(\qq) \simeq H^1(\Gal(\bar{\qq}/\qq), \Pic(Y_{\bar{\qq}})) = 0.$
\item $\Br(Y_{\bar{\qq}}) = \langle \alpha_{\bar{\qq}} \rangle \simeq \zz/2\zz$.
\end{enumerate}
\end{prop}
\begin{proof}\hfill
\begin{enumerate}[(i)] 
\item  By \cite{ht2}*{Proposition 4.2.4}, $Y/\cc$  has Picard rank $1$ and $\Pic^0(Y_\cc)$ is trivial.
The natural map $H^1(\O_{Y_{\bar{\qq}}}) \otimes \cc \to H^1(\O_{Y_{\cc}})$ is an isomorphism by cohomology and base change, hence
$\Pic^0(Y_{\bar{\qq}})$ is also trivial. Thus the statement follows from the fact that the N\'eron-Severi group of a smooth proper variety over an algebraically closed field remains the same after base change to another algebraically closed field \cite{maulikpoonen}*{Proposition~3.1}.
\item The identification $\Br_1(Y)/\Br(\qq) \simeq H^1(\Gal(\bar{\qq}/\qq), \Pic(Y_{\bar{\qq}}))$ follows from the low degree exact sequence of the Hochschild-Serre spectral sequence
 \cite{poonen_2017}*{Corollary~6.7.8} since $H^3(\qq, \gg_m) = 0$ \cite{poonen_2017}*{Remark~6.7.9}.  Since $\Pic(Y_{\bar{\qq}}) = \zz$ and $H^1(\Gal(\bar{\qq}/\qq), \zz) = 0$, the result follows.
\item The natural map of Brauer groups $\Br(Y_{\bar{\qq}})\to \Br(Y_{\cc})$ is an isomorphism by \cite{cts}*{Propositions~5.2.2 and 5.2.3}. 
By Proposition \ref{prop:brauerX} in Appendix \ref{appendix}, $\Br(Y/\cc)\simeq \zz/2\zz$ and is generated by the base change of $\alpha$ to $\cc$, so the result follows.
\qedhere
\end{enumerate}
\end{proof}

\section{An obstruction to weak approximation} \label{sec:obstr_weak}
\subsection{Background}\label{sec:weak_approx}

Recall that a smooth projective variety $Y$ over $\qq$ satisfies \defi{weak approximation} if the $\qq$-rational points are dense inside of the adelic points in the adelic topology. We recall how elements of the Brauer group of $Y$ may obstruct weak approximation. 

Let $L \supset \qq$ be a field.  A point $y \in Y(L)$ gives a map $y \colon \Spec L \to Y$.
By functoriality of Brauer groups, we can pull back an Azumaya algebra $\A$ under $y$ to obtain
\[\A(y) \colonequals y^* \A \in \Br(\Spec L) = \Br(L). \]
We call this pullback \defi{evaluation of $\A$ at $y$}.  When $L = \qq_v$ is the completion of $\qq$ at a place $v$, we may further compose with the invariant map
\[\inv_v \colon \Br(\qq_v) \to \qq/\zz.\]
If $v$ is a finite place, then $\inv_v$ is an isomorphism onto $\qq/\zz$.  For $\infty$ the real infinite place of $\qq$, we have $\inv_\infty \colon \Br(\rr) \xrightarrow{\simeq} (\frac{1}{2}\zz)/\zz \subset \qq/\zz$.
Recall that the fundamental exact sequence of global class field theory \cite{milne_2003}*{Chapter~VIII, Theorem~4.2}
\[0 \to \Br(\qq) \to \bigoplus_v \Br(\qq_v) \xrightarrow{\sum \inv_v} \qq/\zz \to 0\]
says that the tuples of local Brauer classes that come from $\Br(\qq)$ are exactly those whose invariants sum to $0$. For any subset $S\subseteq \Br(Y)$ we can produce an intermediate set of adelic points containing the closure of the $\qq$-rational points 
\begin{equation} \overline{Y(\qq)} \subseteq Y(\aa_\qq)^{S} \subseteq Y(\aa_\qq),\end{equation} 
defined by
\[ Y(\aa_\qq)^S := \Bigl\{ (P_v) \in Y(\aa_\qq) : \sum_v \inv_v \mathcal{A}(P_v) = 0, \text{ for all } \mathcal{A} \in S \Bigr\}.\]
For more details on these inclusions, which arise from class field theory, see \cite{skorobogatov_2001}.

We will produce obstructions to weak approximation by means of the following standard argument.

\begin{lem}\label{lem:real_obstruction}
Let $\A \in \Br(Y)[2]$ and let $v$ be a place of $\qq$.  Suppose that $Y(\aa_\qq) \neq \emptyset$  and that there exist two $\qq_v$ points $Q_v, Q_v' \in Y(\qq_v)$ for which
\[ \inv_v(\A(Q_v)) = 0 \in \Br(\qq_v) \qquad \text{and} \qquad \inv_v(\A(Q_v'))  = \tfrac{1}{2} \in \Br(\qq_v). \] 
Then $Y(\aa_\qq)^\A$ is not equal to $Y(\aa_\qq)$.  In particular, $\A$ obstructs weak approximation on $Y$.
\end{lem}
\begin{proof}
Since $Y(\aa_\qq) \neq \emptyset$, there exist local points $P_\ell \in Y(\qq_\ell)$ for all places $\ell \neq v$.  Since $\A \in \Br(Y)[2]$, \[\sum_{\ell \neq v} \inv_\ell(\A(P_\ell )) \in \frac{1}{2} \zz / \zz.\]
We use these local points to define an adelic point $(P_\ell) \in Y(\aa_\qq)$ where $P_v \in Y(\qq_v)$ is defined by 
\[P_v \colonequals \begin{cases}  Q_v' &\text{if } \ \sum_{\ell \neq v} \inv_\ell(\A(P_\ell )) = 0, \\
Q_v &\text{if } \ \sum_{\ell \neq v} \inv_\ell (\A(P_\ell)) = \frac{1}{2}. \end{cases} \]
By construction, $(P_\ell) \in Y(\aa_\qq) \smallsetminus Y(\aa_\qq)^\A$,  so $\overline{Y(\qq)} \subset Y(\aa_\qq)^\A \subsetneq Y(\aa_\qq)$, and $Y$ fails to satisfy weak approximation.
\end{proof}

\subsection{\boldmath Evaluation of the Brauer class $\alpha$}
\label{expectedbehavior}

We now specialize to the setup of our Calabi--Yau threefold $Y = Y_P$ and $2$-torsion Brauer class $\alpha = \alpha_P \in \Br(Y)[2]$ that are associated to a choice of linear system $P$ of five quadrics in $\pp_{\qq}^4$.  We will always assume that we are working with a regular \(P\) so that \(Y_P\) is smooth.

For any extension $L \supset \qq$, we begin by geometrically describing the evaluation of $\alpha$ at a point $y \in Y(L)$.  By slight abuse of notation, we will write $Q_y$ for the quadric corresponding to the image of $y$ in $H$.  Since $\alpha$ is the Brauer class associated to an \'etale $\pp^1$-bundle $Z$,  the evaluation $\alpha(y)$ is the class of the conic $Z|_y$ over $\Spec L$ in $\Br(L)[2]$.
Therefore 
\begin{align*}
\text{$\alpha(y) = 0 \in \Br(L)[2]$ } &\Leftrightarrow \text{ $Z|_y$ has an $L$-point}\\
&\Leftrightarrow \text{ $Q_{y}$ contains an $L$-rational plane $\Lambda \simeq \pp^2_L$ in ruling $y$.}
\end{align*}

Thus for $y \in Y(\qq_v)$,
\begin{equation}\label{inv_calculation} \inv_v(\alpha(y)) = \begin{cases} 0 
& \text{if $Q_{y}$ contains a $\qq_v$-rational 2-plane in ruling $y$,} \\
\frac{1}{2} & \text{otherwise.}
\end{cases}\end{equation}
We give an equivalent criterion for $\inv_v(\alpha(y))$ to be trivial.

\begin{lem}\label{lem:trivial_evaluation_criterion}
Let $y\in Y(L)$ be such that $Q_{y}$ is of rank at least $3$.  Then $Q_{y}$ contains an $L$-rational $2$-plane in ruling $y$ if and only if $Q_{y}$ has a smooth $L$-point.
\end{lem}
\begin{proof}
To see the forward direction, we note that if $Q_{y}$ contains an $L$-rational plane $\Lambda \simeq \pp^2_L$, then as the dimension of the singular locus of $Q_{y}$ is at most~$1$, $\Lambda$ is generically contained in the smooth locus of $Q_{y}$, which hence has a smooth $L$-point.

In the other direction, first assume that $Q_{y}$ has rank $3$. Then it is a cone (with ``vertex'' a line $\ell$ defined over $L$) over a smooth conic in $\pp_L^2$.  If $Q_{y}$ has a smooth $L$-point, then this conic has an $L$-point, and hence the span of that point and the line $\ell$ is an $L$-rational plane $\Lambda$ in $Q_{y}$ (which has only one ruling).

Similarly, if $Q_{y}$ has rank $4$, then it is a cone (with vertex a point $x \in \pp^4(L)$) over a smooth quadric $\bar{Q}$ in $\pp_L^3$.  If $Q_{y}$ has a smooth $L$-point, then this quadric $\bar{Q}$ has a $L$-point $p$.  The tangent plane section $T_p\bar{Q} \cap \bar{Q}$ consists of two lines, one from each ruling.  As $y$ was an $L$-point of $Y$, these lines are individually defined over $L$, and hence the span of the one in ruling $y$ and $x$ gives the desired plane $\Lambda$ in $Q_{y}$ in ruling $y$.
\end{proof}

Say that a pair $(Y_P, \alpha_P)$ has \defi{trivial evaluation at all finite places} if
\begin{equation}\label{eq:triv_eval}
\text{for all $p$ and all $y \in Y_P(\qq_p)$, we have $\alpha_P(y) = 0.$}\tag{$\dagger$}
\end{equation}
We next show that \eqref{eq:triv_eval} is the ``expected" behavior as we vary the defining linear system \(P\).  
We will guarantee this, using Lemma \ref{lem:trivial_evaluation_criterion} and Hensel's Lemma, by spreading out our family of quadrics to $\Spec \zz$ and showing that over $\ff_p$, all quadrics have smooth $\ff_p$-points. Extra care is necessary when $p=2$ since quadrics and bilinear forms do not coincide in characteristic $2$.

Let  $\sB$ be the universal $5\times 5$ symmetric matrix with even diagonal entries
over $\aa_{\zz}^{15}$.  We say an ideal $I$ of a ring $R$ is \defi{saturated with respect to $s \in R$} if whenever there exists $n  \geq 1$ such that $s^n x \in I$, then $x \in I$. Given an ideal $I$ of a ring $R$ we can produce a new ideal $I'$ that is saturated with respect to $s$ in $R$ through the process of \defi{saturation}: $I' \colonequals \{ x  \in R : s^n x \in I \text{ for some } n \geq 1\}$.
We consider the ideal of $3 \times 3$ minors of $\sB$ and saturate with respect to the element $s = 2$. Define $\I_3$ to be the resulting ideal of $3\times 3$ minors of $\sB$ that is \emph{saturated with respect to $2$}.  Since this is independent of scaling, it descends to the projective space $\pp^{14}$ of quadric hypersurfaces in $\pp^4$.  Let $\V_3 = V(\I_3)$ be the associated subscheme of $\pp_{\zz}^{14}$.  

We can also describe the process of saturation geometrically.  Let \(\V_3^\circ \subset \pp_{\zz[1/2]}^{14}\) be the subscheme of quadrics of rank at most \(2\); the ideal of \(\V_3^\circ\) is given by the \(3 \times 3 \) minors of \(\sB\).  Then the saturation of this ideal at \(2\) is the ideal of the closure \(\V_3 = \overline{\V_3^\circ}\) in \(\pp^{14}_{\zz}\).

\begin{lem}\label{lem:sm_pt_criterion}
Suppose that $Q \subseteq \pp^4_{\ff_q}$ is a quadric whose associated bilinear form $B_Q$ does not lie in $\V_3(\ff_q)$.  Then $Q$ has a smooth $\ff_q$-point.
\end{lem}
\begin{proof}
If $q$ is not a power of $2$, then we may assume that $Q = ax_0^2 + bx_1^2 + cx_2^2 + dx_3^2 + ex_4^2$ is diagonal.  Since $[B_Q] \notin \V_3(\ff_q)$, we may assume that $a$, $b$, and $c$ are nonzero.  The $2$-plane section $x_3=x_4=0$ of $Q$ is therefore a smooth conic, which has smooth $\ff_q$-points by the Chevalley--Warning Theorem. These points are smooth points of $Q$. 

Now, we let $q = 2^r$.  By \cite{albert}*{Theorem~16}, every quadric in $\pp_{\ff_q}^4$ is equivalent to one of the form
\[Q_0 = ax_0^2 + bx_1^2 + cx_2^2 + dx_3^2 + ex_4^2 + f x_0x_1 + gx_2x_3. \]

Working over \(\zz\), this formula defines a map from the parameter space \(\aa_\zz^{15}\) of quadrics (or symmetric bilinear forms with even diagonal entries) to \(\aa_\zz^7\) with variables \(a, b, c, d, e, f, g\).  Let \(I \subseteq \zz[a,b,c,d,e,f,g]\) denote the image of the ideal \(\I_3\) under this specialization.  (Since pullback and Zariski closure do not commute, this is \emph{not} the same as the ideal of $3 \times 3$ minors of $B_{Q_0}$ saturated at $2$.)

We now specialize to characteristic \(2\).  Write $\bar{I} = I \mod 2$.  Then one can compute that explicitly that \(\bar{I}\) has generators
\[\bar{I} = (f^2c, f^2d, f^2e, f^2g, g^2a, g^2b, g^2e, g^2f).\]
It therefore suffices to show that if $a, b, \ldots, g \in \ff_q$ is a specialization such that one of the above generators of $\bar{I}$ is nonzero in $\ff_q$, then that associated quadric has a smooth $\ff_q$-point. 
We consider the generators in turn below.
\begin{enumerate}[{{Case}}~1.]
\item Generators guaranteeing both $f,g \in \ff_q^\times$ (i.e., generators $f^2g$, $g^2f$).  In this case, the hyperplane section $x_4=0$ is a smooth quadric surface, and hence by Chevalley-Warning, has smooth $\ff_q$-points.
\item Generators guaranteeing both $e,f \in \ff_q^\times$ or both $e,g \in \ff_q^\times$ (i.e., generators $f^2e$, $g^2e$).  In this case, the $2$-plane section $x_2=x_3=0$ (resp.~$x_0=x_1=0$) is a smooth conic, and hence by Chevalley-Warning, has smooth $\ff_q$-points.
\item Generators guaranteeing both $f$ and $c$ or $d \in \ff_q^\times$ or both $g$ and $a$ or $b$ $\in \ff_q^\times$ (i.e., generators $f^2c$, $f^2d$, $g^2a$, $g^2b$).  In this case, the $2$-plane section $x_3=x_4=0$ or $x_2=x_4=0$ (resp.~$x_1=x_4=0$ or $x_0=x_4=0$) is a smooth conic, and hence by Chevalley-Warning, has smooth $\ff_q$-points. \qedhere
\end{enumerate}
\end{proof}

The key technical ingredient in understanding the density of planes satisfying \eqref{eq:triv_eval} is the \defi{Ekedahl sieve}.  Recall that the height of an integral point on \(\aa^n\) is the maximum of the absolute value of the coordinates:
\[h((p_1, \dots, p_n)) = \max_i |p_i|.\]
For compactness of notation we will write \(\aa^n(\zz)_{\leq N}\) for the set of integral points on \(\aa^n\) of height at most \(N\).

\begin{lem}[The Ekedahl sieve]\label{lem:eke}
Let \(Z \subset \aa^n\) be a subscheme of codimension at least \(2\).  Then we have
\[ \lim_{M \to \infty} \lim_{N \to \infty} \frac{\#\{P \in \aa^n(\zz)_{\leq N} :  \exists p > M \text{ s.t. } (P \mod p) \in Z(\ff_p)\}}{\#\aa^n(\zz)_{\leq N}} =  0.\]
\end{lem}
\begin{proof}
Since any subscheme of codimension at least \(2\) is contained in a complete intersection subscheme of codimension \(2\), the result follows from \cite[Lemma 21]{poonen_stoll}.  See also \cite[Theorem 1.2]{ekedahl}.
\end{proof}

This result allows one to compute the global density of points that satisfy ``codimension at least \(2\)'' local conditions at all primes as a product of local densities.

As above, let \(Z \subset \aa^n\) be a subscheme of codimension at least \(2\).  Let \(S_p \subset Z(\ff_p)\) be a subset of the \(\ff_p\) points of \(Z\).  Write \(S = (S_p)_{p \text{ prime}}\).  Define
\[d(!S) = \lim_{N \to \infty} \frac{\#\{P \in \aa^n(\zz)_{\leq N} :  \forall p \  (P \mod p) \not\in S_p\}}{\# \aa^n(\zz)_{\leq N} }.\]

\begin{prop}\label{prop:density_product}
With the notation and assumptions as above,
\[ d(!S) = \prod_{p} \left( 1 - \frac{\#S_p}{\#\aa^n(\ff_p)} \right).\]
\end{prop}

\begin{rmk}
The Lang-Weil estimates guarantee that the infinite product on the right converges.
\end{rmk}

\begin{proof}
Let \(M\) be an integer.  We have bounds
\begin{align*} 
& \frac{\#\{P \in \aa^n(\zz)_{\leq N} :  \forall p\leq M \  (P \mod p) \not\in S_p\}}{\#\aa^n(\zz)_{\leq N}} \\
& - \frac{\#\{P \in \aa^n(\zz)_{\leq N} :  \exists p > M  \text{ s.t. }  (P \mod p) \in S_p\}}{\#\aa^n(\zz)_{\leq N}} \\
& \qquad \qquad \leq  \frac{\#\{P \in \aa^n(\zz)_{\leq N} :  \forall p \  (P \mod p) \not\in S_p\}}{\#\aa^n(\zz)_{\leq N}} \\
 & \qquad \qquad \qquad \leq \frac{\#\{P \in \aa^n(\zz)_{\leq N} :  \forall p \leq M \  (P \mod p) \not\in S_p\}}{\#\aa^n(\zz)_{\leq N}}.
 \end{align*}
 Taking the limit as \(N \to \infty\) (each term of which exists) we obtain:
 \begin{align*} 
& \lim_{N \to \infty}\frac{\#\{P \in \aa^n(\zz)_{\leq N} :  \forall p\leq M \  (P \mod p) \not\in S_p\}}{\#\aa^n(\zz)_{\leq N}} \\
&-  \lim_{N \to \infty}\frac{\#\{P \in \aa^n(\zz)_{\leq N} :  \exists p > M  \text{ s.t. }  (P \mod p) \in S_p\}}{\#\aa^n(\zz)_{\leq N}} \\
& \qquad \qquad \leq  d(!S) \\
 & \qquad \qquad \qquad \leq  \lim_{N \to \infty}\frac{\#\{P \in \aa^n(\zz)_{\leq N} :  \forall p \leq M \  (P \mod p) \not\in S_p\}}{\#\aa^n(\zz)_{\leq N}}.
 \end{align*}
Since \(S_p \subseteq Z(\ff_p)\), by Lemma \ref{lem:eke}, we have
\[ \lim_{M \to \infty} \lim_{N \to \infty} \frac{\#\{P \in \aa^n(\zz)_{\leq N} :  \exists p > M \text{ s.t. } (P \mod p) \in S_p)\}}{\#\aa^n(\zz)_{\leq N}} =  0.\]
Therefore
\[d(!S) =  \lim_{M \to \infty}\lim_{N \to \infty} \frac{\#\{P \in \aa^n(\zz)_{\leq N}, \forall p \leq M\  (P \mod p) \not\in S_p\}}{\#\aa^n(\zz)_{\leq N}}.\]
By the Chinese Remainder Theorem we have that
\[ \lim_{N \to \infty} \frac{\#\{P \in \aa^n(\zz)_{\leq N}, \forall p \leq M\  (P \mod p) \not\in S_p\}}{\#\aa^n(\zz)_{\leq N}} = \prod_{p \leq M} \left( 1 - \frac{\#S_p}{\#\aa^n(\ff_p)} \right).\]
The result follows by taking the limit as \(M \to \infty\).
\end{proof}

Proposition \ref{prop:density_product} formally implies the same result where we replace \(\aa^n\) by an open locus in \(\aa^n\) whose complement has codimension at least \(2\) (divide the density coming from the local conditions \(S\) in all of \(\aa^n\) by the density coming from the local conditions of being in the complement of the open locus).

\medskip

We now apply this to our situation.  To obtain an affine space and a codimension at least \(2\) subscheme, we will consider the \defi{frame space} \(\Fr(k, n)\) parameterizing bases of \((k+1)\)-dimensional subspaces of an \((n+1)\)-dimensional vector space.  Explicitly, \(\Fr(k,n)\) is the complement of the vanishing of the \((k+1)\)-minors in space of \((k+1)\times (n+1)\)-matrices, which is itself \(\aa^{(k+1)(n+1)}\).  If \( k< n\) then the subscheme where all \((k+1)\)-minors vanish is codimension at least \(2\).

On the other hand, the natural map
\[\pi \colon \Fr(k,n) \to \Gr(k,n)\]
sending a basis to its span exhibits \(\Fr(k,n)\) as a \(\GL_{k+1}\)-torsor over \(\Gr(k,n)\).  For this reason, for any subscheme \(W \subseteq \Gr(k,n)\), we have
\[\frac{\#W(\ff_p)}{\#\Gr(k,n)(\ff_p)} = \frac{\#\pi^{-1}W(\ff_p)}{\#\Fr(k,n)(\ff_p)}.\]

The $4$-planes that meet $\V_3$ form a Zariski-closed subscheme \(\W\) of $\gg_{\zz}(4, 14)$ of codimension $2$.  Since \(\pi^{-1} \W\) has codimension \(2\) in \(\Fr(4, 14)\),
we can use Proposition \ref{prop:density_product} to compute the density of planes \([P] \in \Gr(4,14)(\qq)\) that give rise to pairs $(Y_P, \alpha_P)$ that satisfy
\eqref{eq:triv_eval}.  Recall from the introduction that we define the density by sampling integral points on \(\Fr(4,14)\) of bounded height going to infinity.  This notion of density, therefore, agrees with the one used in Proposition \ref{prop:density_product} for subschemes of affine space.

We will write \(\sP\) to indicate the linear system \(P\) over \(\zz\).
Let $S_p \subseteq \gg(4, 14)(\ff_p)$ denote the subset of planes \(P_p \in \gg(4, 14)(\ff_p)\) such that there exists a point in $P_p(\ff_p)$ for which the corresponding quadric \emph{does not} have a smooth $\ff_p$-point.  By Lemma \ref{lem:sm_pt_criterion}, we have that \(S_p \subseteq \W(\ff_p)\).  (In particular, if the $4$-plane $\sP$ and $\V_3$ do not meet in $\pp^{14}_\zz$, then for every prime \(p\), we have that $[\sP \mod p]$ is not in \(S_p\).  Since it is easy to compute with the ideal $\I_3$, in practice we will use this criterion to rule out membership in $S_p$ for all \(p\) simultaneously.)

\begin{prop}\label{prop:eval_finite_primes}
The set of integral points of \(\Fr(4, 14)\) whose reductions are not contained in \(\pi^{-1}(S_p)\) for any \(p\) have density 
\[d(!S) =   \prod_p \left( 1 - \frac{\#S_p}{\#\Gr(4, 14)(\ff_p)} \right) \geq .73.\]
\end{prop}
\begin{proof}
The left equality follows from Proposition \ref{prop:density_product}.  It suffices now to compute a lower bound on this infinite product.
Let $B_p \subset \pp^{14}(\ff_p)$ denote the set of quadrics over $\ff_p$ that \emph{do not} have a smooth $\ff_p$-point.   
This is exactly the union of the set of double planes and the set of quadratic conjugate planes (when \(p> 2\) these are the quadrics of rank \(1\) and non-split quadrics of rank $2$) .  The set of double planes has cardinality $\#\gg(3, 4)(\ff_p)$.  Similarly, a pair of quadratic conjugate planes is a cone with vertex a codimension \(2\) plane over a nonsplit quadric in \(\pp^1\).  Fixing the vertex, we may assume that it is of the form $x_0^2 - bx_0x_1 + c x_1^2$ for some $b, c \in \ff_p$.  If \(p = 2\), such a quadric is nonsplit if and only if \(b = c= 1\).  If \(p > 2\), such a quadric is nonsplit if and only if $\Delta = b^2 - 4c$ is not a square in $\ff_p^\times$.  There are therefore $(p^2-p)/2$ such choices of $b$ and $c$.  In total, we have
\[\#B_p = \#\gg(3, 4)(\ff_p) + \#\gg(2, 4)(\ff_p) \cdot (p^2-p)/2.\]

Finally, the number of $4$-planes over $\ff_p$ that contain an $\ff_p$-point in $B_p$ is at most $\#\gg(3, 13)\#B_p$.  Therefore the fraction of $4$-planes in $\gg(4, 14)(\ff_p)$ that contain an $\ff_p$-point in $B_p$ is at most
\[ b(p) \colonequals \frac{\#\gg(3, 13)(\ff_p)\#B_p}{\#\gg(4, 14)(\ff_p)} =  \frac{p^8 + p^6 + 2p^4 + p^3 + 2p^2 + p + 2}{2p^{10} + 2p^5 + 2} . \]
The function $p^2b(p)$ is monotonically decreasing for $p >2$ and has an asymptote at $1/2$.  Its value on the interval of $p \geq N$ is therefore bounded by
\begin{equation}\label{bound}
\max(N^2b(N), 1/2)
\end{equation} 
We have 
\begin{align*}
d(!S) & = \prod_p \left( 1 - b(p) \right)\\
&\geq \left(\prod_{p < 100} \left( 1 - b(p) \right)\right) \cdot \left( 1 - \sum_{p > 100} b(p) \right)
\intertext{One can compute (e.g., using the functions in the file \texttt{probability.sage}) that the first factor is at least $.737$.  Combining this with \eqref{bound},}
&\geq (.737)\left(1 - \sum_{p > 100} \frac{.50005}{p^2} \right) \\ 
&\geq (.737)\left(1 - (.50005) \int_{100}^\infty \frac{1}{p^2} dp \right) \\
&\geq .73. \qedhere
\end{align*}
\end{proof}

\begin{prop}\label{prop:fam_sm_pts}
Suppose that $(Y_P, \alpha_P)$ is defined by a regular $4$-dimensional linear system  $[P] \in \gg(4, 14)(\qq) = \gg(4, 14)(\zz)$ that is not contained in the set \(S_p \subset \gg(4, 14)(\ff_p)\) for any \(p\).  Then for any $\qq_p$-point $y \in Y_P(\qq_p)$,
the associated quadric $Q_y$ has a smooth $\qq_p$-point.
\end{prop}
\begin{proof}
We spread out $P\simeq \pp^4_\qq$ to $\sP \simeq \pp^4_{\zz}$.  By slight abuse of notation, we also write $y \in P(\qq_p)$ for the image of $y \in Y(\qq_p)$.
This spreads out to a map $\Spec(\zz_p) \to \sP$ by properness of $\sP$ over $\Spec(\zz)$, and therefore gives a quadric $\sQ$ over $\Spec(\zz_p)$ with generic fiber $Q_y$ and special fiber $\bar{Q}$ a quadric over $\ff_p$.  Since $[P] \not\in S_p$, this quadric $\bar{Q}$ has a smooth $\ff_p$-point.  Therefore, the smooth locus of $\sQ \to \Spec(\zz_p)$ has $\ff_p$-points, and hence by Hensel's Lemma 
\cite{poonen_2017}*{Theorem~3.5.63} 
it has $\zz_p$-points.  Therefore $Q_y$ has smooth $\qq_p$-points.  
\end{proof}

\begin{cor}\label{cor:expected_behavior}  A randomly generated integral basis yields a span \(P\) 
for which the associated pair $(Y_P, \alpha_P)$ satisfies trivial evaluation at all finite places \eqref{eq:triv_eval} with probability at least \(.73\).
\end{cor}
\begin{proof}
Since regularity is an open condition, a random linear system $P$ will be regular with probability $1$.  Hence with probability at least $.73$, it is a regular linear system with $[P \mod p] \not\in S_p$ for all \(p\).
By Proposition \ref{prop:fam_sm_pts}, for every $y \in Y_P(\qq_p)$, the corresponding quadric $Q_y$ has a smooth $\qq_p$-point.
Since regularity implies that $P$ avoids the locus of quadrics of rank $2$ or less, it now follows by Equation \eqref{inv_calculation} and Lemma \ref{lem:trivial_evaluation_criterion} that $\alpha_P(y)=0$, as desired.
\end{proof}

\subsection{Proofs of the main theorems} \label{sec:mainproof} In this section, we prove the main theorems guaranteeing that $Y_P$ has a Brauer--Manin obstruction to weak approximation.

\begin{proof}[Proof of Theorem \ref{thm:main}]
 Let $Y = Y_P$ be the Calabi--Yau threefold with class $\alpha = \alpha_P\in \Br(Y_P)[2]$ that is associated to the choice of a regular $4$-plane $P$ in $\pp^{14}_{\qq}$.  By assumption, there exists a real point $h \in H(\rr)$ so that the corresponding quadric $Q_h$ has signature $(+1, +1, +1, +1, 0)$ or $(-1, -1, -1, -1, 0)$.  As such, $Q_h$ is the cone over a smooth quadric $\bar{Q}_h$ in $\pp^3_\rr$ and the discriminant of $\bar{Q}_h$ is a square in $\rr$.  Therefore the preimage of $h$ in $Y$ is two points $y_1$ and $y_2$ in $Y(\rr)$.  However, the quadric $Q_h$ has no smooth $\rr$-points, since $\bar{Q}_h(\rr) = \emptyset$.  Therefore, by Lemma \ref{lem:trivial_evaluation_criterion} in combination with \eqref{inv_calculation}, we have that
\[\alpha(y_1) = \alpha(y_2) = \tfrac{1}{2} \in \Br(\rr).\]
However, by Lemma \ref{lem:evaluation_real}, there must exist a point $y_0 \in Y(\rr)$ such that $Z|_q(\rr) \neq \emptyset$.  Again by \eqref{inv_calculation}, we have
\[\alpha(y_0) = 0 \in \Br(\rr).\]
By Lemma \ref{lem:real_obstruction}, this implies that $Y$ fails to satisfy weak approximation.
\end{proof}

\begin{proof}[Proof of Corollary \ref{cor:easy}]
The quadric $Q_0:=x_0^2 + x_1^2 + x_2^2 + x_3^2 = 0$ defined over $\qq$ is the cone over a smooth quadric surface of discriminant $1$.  Therefore, it gives rise to a $\qq$-point on $Y_P$ for any linear system \(P\) containing \([Q_0]\).  Furthermore, over $\rr$, this quadric has signature $(+1, +1, +1, +1, 0)$.
Finally, the condition of regularity (Definition \ref{def:reg}) is open; therefore, it suffices to show that one such $4$-plane through $[Q_0]$ is regular.

We verified the regularity of  the  $4$-plane $P$ given by the following additional generating quadrics with  Magma, using code in the ancillary file \texttt{yexample.magma}, located in the repository \cite{githubrepo} and in the supplementary files.
\begin{align*}
&    x_0x_1 + x_0x_4 + x_1x_3 - x_1x_4 + x_2x_3 + x_3x_4,\\
&    x_0^2 - x_0x_2 - x_0x_4 + x_1x_2 - x_1x_3 - x_1x_4 - x_2^2 - x_2x_4 - x_3x_4 - x_4^2,\\
&    x_0x_1 + x_0x_2 + x_0x_4 + x_1^2 - x_1x_4 + x_2^2 + x_2x_3 - x_2x_4 + x_3^2,\\
&    x_0^2 + x_0x_1 + x_0x_2 - x_0x_4 - x_1^2 - x_1x_2 + x_1x_3 - x_1x_4 - x_2x_3 - x_2x_4+ x_4^2.\hfill\qedhere
  \end{align*}
\end{proof}

\subsection{Explicit example}\label{sec:example}

In this section we prove Theorem \ref{thm:example}, explicitly exhibiting a choice of quadrics \(P\) for which the associated Brauer class $\alpha_P$ on the associated Calabi--Yau threefold $Y_P$ obstructs weak approximation. Recall that $Q_0, Q_1, Q_2, Q_3,$ and $Q_4$ are given by the equations
\begin{align*}
&x_0x_1 + x_2x_3\\
&x_0^2 + x_0x_1 + 3x_0x_2 + 2x_1^2 + x_1x_2 + x_1x_3 + 5x_2^2 + x_2x_3 + x_3^2\\
&x_0^2 + 2x_0x_2 + x_0x_3 + x_0x_4 + x_1x_2 + x_1x_4 + x_2x_4 + 4x_3x_4\\
&4x_0^2 + 4x_1^2 + 3x_1x_3 + 2x_1x_4 + 3x_2^2 + x_2x_4 + x_4^2\\
&x_0^2 + x_0x_2 + x_0x_3 + 4x_0x_4 + 3x_1x_2 + 4x_1x_4 + 2x_2x_3 + 2x_3^2 + x_4^2.
\end{align*}
Furthermore, this example will exhibit the ``expected" behavior of Corollary~\ref{cor:expected_behavior}: the pair $(Y_P, \alpha_P)$ has constant trivial evaluation at all finite primes \eqref{eq:triv_eval}.  

\begin{proof}[Proof of Theorem \ref{thm:example}]
Corollary \ref{cor:expected_behavior} proves the first part of the Theorem.
   
Let $Q_0, \dots, Q_4$ be the explicit choice of linearly independent quadrics
given above.  The quadrics $Q_0, \dots, Q_4$ define a regular system $P$, which can be verified using
using Magma code included in the ancillary file \texttt{yexample.magma}. 
Write \((Y, \alpha) = (Y_P, \alpha_P)\).  We first show the containments
\[\emptyset \neq Y_P(\qq) \subset Y_P(\aa_{\qq})^{\alpha_P} \subsetneq Y_P(\aa_{\qq}).\]
The quadric $Q_0$ is rank $4$, and is a cone with vertex $[0:0:0:0:1]$ over the smooth quadric 
\[\bar{Q}_0 \colonequals x_0x_1 + x_2x_3\]
in $\pp_\qq^3$.  Since $\disc(\bar{Q}_0) = 1$ is a square, the two rulings of lines on $\bar{Q}_0$ are defined over $\qq$.  Therefore the fiber over $[Q_0] \in H(\qq)$ in $Y$ consists of two $\qq$-points, and so $Y(\qq)  \neq \emptyset$ (and, in particular, \(Y\) is everywhere locally soluble).  Furthermore, $Q_1 \colonequals x_0^2 + x_0x_1 + 3x_0x_2 + 2x_1^2 + x_1x_2 + x_1x_3 + 5x_2^2 + x_2x_3 + x_3^2$ has signature $(+1,+1,+1, +1, 0)$; therefore this choice of $P$ satisfies the conditions of Theorem \ref{thm:main}, and $\alpha$ obstructs weak approximation on $Y$.

Finally, we show that $(Y_P, \alpha_P)$ has constant trivial evaluation at all finite primes \eqref{eq:triv_eval}.
Let $Q$ be the universal quadric $Q = t_0 Q_0 + \dots + t_4 Q_4$ over the projective space $P \simeq \pp^4_{\zz}$. 
Let $J_3$ be the ideal of $\zz[t_0, \dots, t_4]$ given by specializing $\I_3$ to $Q$ and then saturating with respect to the ideal $(t_0, \dots, t_4)$. A Groebner basis calculation over $\zz$ shows that $1$ is in $J_3$, and therefore for the choice of quadrics above, we have $P \cap \V_3 = \emptyset$.  Therefore the reduction of every quadric in $P$ (and hence \(Y\)) modulo every prime \(p\) has a smooth \(\ff_p\) point by Lemma \ref{lem:sm_pt_criterion}.  Hence \([P \mod p]\) is not contained in the set \(S_p\) for any prime \(p\).
Thus, by Proposition \ref{prop:fam_sm_pts}, for every $\qq_p$-point $y \in Y(\qq_p)$, the associated quadric $Q_{y}$ has a smooth $\qq_p$-point.  Hence by Lemma \ref{lem:trivial_evaluation_criterion} and equation \eqref{inv_calculation}, $y^*\alpha = \alpha|_y =0 \in \Br(\qq_p)[2]$.

\end{proof}

\subsection{Obstructions from finite primes}

\begin{prop}\label{prop:q3}
Let $P  \subset \pp H^0(\pp_{\qq}^4, \O_{\pp^4}(2))$ be a general $4$-plane through the points $[Q_0 \colon 6x_0^2 -3x_1^2 + 2x_2^2 - x_3^2 = 0]$ and
  $[Q_1 \colon x_0x_1 + x_2x_3= 0]$
in $\pp^{14}(\qq)$.  Then $\emptyset \neq \overline{Y_P(\qq)} \subsetneq Y_P(\aa_\qq)$ and so $Y_P$ does not satisfy weak approximation.
\end{prop}
\begin{proof}
The condition of regularity (Definition \ref{def:reg}) is open.

The quadric $Q_0$ does not contain any smooth $\qq_3$-points, as we now show. We may rewrite the equation for $Q_0$ as
$3(2x_0^2-x_1^2)+(2x_2^2-x_3^2)=0$, and directly check modulo $3$ we only have the non-smooth point $(0:0:0:0:1)$. Consider solutions to $Q_0$ mod $3^n$ for $n \in \nn$; the valuation of $x_i$ in the solution must increase as $n$ does, for $i = 0, \dots , 3$.

However, $Q_0$ is the cone over a smooth quadric surface with discriminant~$36$, so $Y_P$ has at least one $\qq_3$-point over $[Q_0] \in H(\qq_3)$.

The quadric $Q_1$ contains $\qq$-points 
and, over $\qq$, is the cone over a smooth quadric surface of discriminant $1$.  Therefore, it also contains smooth $\qq_3$-points and gives rise to a $\qq_3$-point on $Y_P$. Thus, by Lemma \ref{lem:real_obstruction}, for such a choice of regular $4$-plane $P$, the Brauer class $\alpha_P$ will obstruct weak approximation.

To finish, it suffices to produce a regular $4$-plane $P$ containing $[Q_0]$ and $[Q_1]$.
The regularity of  the  $4$-plane  given by the following additional generating quadrics was verified with Magma, using code in the ancillary file \texttt{yexample.magma}.
\begin{align*}
 x_0^2 &- x_0x_2 - x_0x_3 - x_1^2 + x_1x_2 - x_2^2 - x_2x_4 + x_3^2 - x_3x_4 - x_4^2,\\
    x_0^2 &+ x_0x_1 + x_0x_2 - x_0x_3 + x_0x_4 - x_1^2 + x_1x_2 - x_1x_3 + x_2x_3 - x_2x_4,\\
        &- x_3^2 - x_3x_4,\\
    x_0^2 &- x_0x_3 + x_1^2 + x_1x_2 + x_1x_3 + x_2^2 + x_3x_4 - x_4^2.\hfill\qedhere
\end{align*}
\end{proof}

\begin{rmk}\label{rmk:rational}
\hfill
\begin{enumerate}[(a)]
\item  It is possible that the example produced in the proof of Proposition~\ref{prop:q3} satisfies the hypotheses of Theorem~\ref{thm:main}. 
It has been chosen so that none of its generators, considered over $\rr$, have signature $(+1,+1,+1,+1,0)$ or $(-1,-1,-1,-1,0)$, but  some element of $P$ may correspond to a quadric which does.
  
\item  One can also ask if there are examples of linear systems \(P\) where $\alpha_P$ obstructs the Hasse principle on $Y_P$, i.e., where $\alpha_P$ has constant evaluation at all $\qq_v$-points for each place $v$, but the sum of the invariants over all places is nontrivial.  It appears that $P$ would have to be incredibly special in order for this to occur.  First, by Lemma \ref{lem:evaluation_real}, there always exists a real point of $Y_P$ where $\alpha_P$ evaluates to $0$, and so evaluation at real points would need to be trivial.  Therefore, there must exist an odd number of finite places~$p$ for which the evaluation at all $\qq_p$-points is constant and nontrivial. However, by Lemmas \ref{lem:trivial_evaluation_criterion} and \ref{prop:fam_sm_pts}, this only occurs if every such quadric over $\qq_p$ (which is necessarily of rank at least $3$) reduces to a quadric with rank at most $2$ over $\ff_p$.
\end{enumerate}  
\end{rmk}

\section{The derived equivalent threefold $X$} \label{sec:derivedequiv}

Hosono and Takagi show in \cite{ht3} that a
choice of regular linear system $P$ as in Definition \ref{def:reg} can also be used to construct another Calabi--Yau variety they call $X = X_P$.
Hosono and Takagi proved in \cite{ht1}*{Theorem~8.0.2}  that, for any fixed choice of regular linear system $P$ over $\cc$, the complex varieties $X_P$ and $Y_P$ are derived equivalent but non-birational over $\cc$. That is, there is a $\cc$-linear equivalence between their bounded derived categories of coherent sheaves.
See \cite{FMTransforms} or \cite{caldararu} for an introduction to these categories.

In many cases, derived equivalence is a strong relationship. For example, a smooth curve is determined uniquely by its associated derived category. 
Examples of varieties where some invariant is not preserved under such an equivalence is considered very interesting. Hosono and Takagi's Calabi--Yau threefolds are one of very few known examples
 of derived equivalent varieties that do not have the same fundamental group \cite{ht1} or Brauer group \cite{br_not_inv_HHLV}. 

There has been recent interest in comparing
the rational points 
of derived equivalent varieties; see \cite{HT}.
Work of Addington, Antieau, Frei and Honigs \cite{aafh} gives
the first examples of derived equivalent varieties where one variety has a $\qq$-point and the other does not. The varieties in these examples are all either hyperk\"ahler fourfolds or abelian varieties of dimension at least $2$.
Because Hosono and Takagi's Calabi--Yau threefolds are already known to not share some invariants, particularly the Brauer group, they seem a likely place to look for such behavior occurring in a different class of varieties. However, to the extent that we are able to compare them, we find the same behavior in both $X$ and $Y$ with regard to $\qq$-rational points. This raises several questions: Is there some special case in which these examples do not share the same behavior with regard to $\qq$-points? If not, is derived equivalence a more restrictive condition on Calabi--Yau threefolds with regard to $\qq$-rational points than it is on the types of varieties in the examples of \cite{aafh}? If so, what is the mechanism?

In this section we give the construction of $X_P$ and show that if $X_P$ and $Y_P$ are defined over $\qq$, then Hosono and Takagi's derived equivalence is also defined over $\qq$. 
Then, in order to compare $X_P$ with $Y_P$,  we show that for any regular $P$, if $X_P(\qq)\neq\emptyset$ then $X_P$ does not satisfy weak approximation. Although $X_P$ and $Y_P$ do not have equal Brauer groups or fundamental groups, in all examples we have analyzed,  $X_P$ and $Y_P$ both have $\qq$-points and do not satisfy weak approximation.

\subsection{The variety $X$}\label{sec:defx}
Consider a choice of regular linear system $P$ over a field $K$
generated by quadrics $Q_0,\ldots,Q_4\subseteq\pp^4$.
When the characteristic of $K$ is not $2$, the data of each quadric $Q_i$ is equivalent to the data of a projective symmetric bilinear form $B_i$, and hence a $(1,1)$-divisor  in $\pp^4 \times \pp^4$ that is symmetric under the $\zz/2\zz$-action swapping the factors.  Let $\tilde{X}_P$ denote the complete intersection of $B_0, \dots, B_4$.
By the regularity of $P$, $\tilde{X}_P$ does not meet the diagonal in $\pp^4 \times \pp^4$,  
and hence the $\zz/2\zz$-action induces
a fixed-point-free involution $\iota \colon \tilde{X}_P \to \tilde{X}_P$.  In these terms, $X_P$ is the quotient by this involution:
\[X_P = \tilde{X}_P / \iota. \]
The smoothness of $X_P$ is equivalent to the regularity of $P$
 \cite{ht3}*{Proposition~2.1}.

\subsection{The kernel of the equivalence}

For smooth projective varieties $X$ and $Y$,
by a result of Orlov,
equivalences between their associated derived categories are produced by \defi{Fourier-Mukai transforms} \cite[Thm~5.14]{FMTransforms}. Such a transform is defined via an object in the derived category of $X\times Y$, which we call a \defi{Fourier-Mukai kernel}. The interested reader is directed to \cite[Chapter~5]{FMTransforms} for a precise discussion on the construction of these functors. 

\begin{lem}
For any choice of regular system $P$ defined over $\qq$, there is a $\qq$-linear derived equivalence between the associated varieties $X_P$ and $Y_P$.  
\end{lem}

\begin{proof}
Write \(X = X_P\) and \(Y = Y_P\).
 As shown in Section~\ref{sec:Y}, there is a natural projection map from $Z \to \gg(2, \pp^4)$.
Since each point in $X$ corresponds to a pair of distinct points in $\pp^4$,
there is a natural map $X\to \gg(1, \pp^4)$ that sends
each point in $X$ to the line in $\pp^4$ containing the corresponding  two points in $\pp^4$.
Consider the subscheme $\Delta\subseteq X\times Y$ formed by pulling back the flag variety $F(1,2,V)\subseteq \gg(1, \pp^4)\times\gg(2, \pp^4)$ to $X\times Z$ and then taking its image in $X\times Y$. In $X\times Z$,
points of this subscheme correspond to the data of 
a quadric in $P$, a plane it contains, and two  
points in $\pp^4$
which lie on that plane.

Hosono and Takagi show in \cite[Theorem~8.0.1]{ht1} that, when working over $\cc$, the ideal sheaf of $\Delta$ is the kernel of a Fourier-Mukai equivalence.
Just as we may define $X,Y$ over $\qq$ in a way that is compatible with base change, we see from our description of $\Delta$ that it may be defined over $\qq$ as well.
To finish, we observe that if
two smooth, projective varieties and a kernel of a Fourier-Mukai functor between their bounded derived categories of coherent sheaves
are defined over a field $K$ and the functor is shown to give a derived equivalence  after base change to $\bar{K}$, then it is an equivalence over $K$ \cite[cf.~proof of Theorem 4]{aafh}. 
\end{proof}

\subsection{Rational points and weak approximation on $X$}

By the following lemma, any linear system $P$ satisfying the hypotheses of Theorem~\ref{thm:main} or Proposition~\ref{prop:q3}
gives rise to a variety $X_P$ where $X_P(\qq)\neq\emptyset$ and therefore $X_P$ does not satisfy weak approximation.

\begin{lem}
If $X_P(\aa_\qq) \neq \emptyset$, then  $X_P$ does not satisfy weak approximation.  
\end{lem}

\begin{proof}
As in \cite{ht2}*{Proposition~3.5.3}, the Lefschetz hyperplane theorem guarantees that $\pi_1(\tilde{X}_P) = 0$, and so 
the \'etale double cover $\pi \colon \tilde{X}_P \to X_P$ represents the nontrivial element of $\pi_1(X_P) \simeq \zz/2\zz$. 

By \cite{Minchev}*{Theorem~2.4.4}, this shows $X_P$ does not satisfy weak approximation.
\end{proof}  

Recall that \(H \subset P\) is the discriminant locus of singular quadrics in the linear system \(P\).

\begin{lem} \label{lem:waX}
Given any choice of regular linear system $P$ defined over $\qq$, $H(\qq)\neq\emptyset$ if and only if $X_P(\qq)\neq\emptyset$. 
\end{lem}

\begin{proof}
 Fix a regular linear system $P$ generated by quadrics defined over $\qq$ corresponding to bilinear forms $B_0,\ldots,B_4$.

Suppose $H(\qq)\neq\emptyset$.
Suppose $B_0$ has rank less than $5$. Then we may choose a point $v\in\pp^4_{\qq}$  in its kernel. If we pick a representative $v'$ of $v$ in $\qq^5$, we note $B_0v'=0$ and $\{B_0v',B_1v',B_2v',B_3v',B_4v'\}$
span a subspace of $\qq^5$ of dimension at most $4$. Therefore, we may pick a nonzero vector with coefficients in $\qq$ in the orthogonal complement, giving an element
$w\in\pp^4_{\qq}$ so that 
 $wB_iv=0$  for all $i$, hence the unordered pair $(w,v)$ determines a point on $X_P$. 
For instance, 
for the choice of quadrics in Theorem~\ref{thm:example},
$([0:0:0:0:1],[1:-3:2:4:1])$ is a $\qq$-point of the corresponding variety $X_P$.

If all of the bilinear forms $B_0,\ldots,B_4$ have full rank, then
by the assumption that $H(\qq)$ is not empty,
we may choose a different basis for our linear system that includes a matrix that does not have full rank. Hence,~$X_P(\qq)\neq\emptyset$.

Now suppose $X_P(\qq)\neq\emptyset$. Then there exist $v,w\in\pp^4_{\qq}$ so that
$wB_iv=0$ for all $i$. The five vectors $B_iv$ are contained in the $4$-dimensional space $w^{\perp}$ and are therefore linearly dependent. Thus there exists a linear combination of the $B_i$ where $v$ is contained in its kernel, and so it does not have full rank.
\end{proof}

\newpage 
\appendix

\section{The Brauer group of \\Hosono and Takagi's threefold\\ by Nicolas Addington}
\label{appendix}

In \cite{ht_der_eq}, Hosono and Takagi constructed a derived equivalence between a pair of Calabi--Yau threefolds $X$ and $Y$ with $\pi_1(X) = \zz/2$ and $\pi_1(Y) = 0$, and an explicit Brauer class of order 2 on $Y$.  In \cite{br_not_inv} I showed that this derived equivalence and a little topological K-theory yield an exact sequence
\[ 0 \to \zz/2 \to \Br(Y) \to \Br(X) \to 0. \]
The purpose of this appendix is to prove:
\begin{prop}\label{prop:brauerX}
$\Br(X) = 0$, and hence $\Br(Y) = \zz/2$.
\end{prop}

First recall the construction of $X$.  We choose 5 symmetric bilinear forms on $\cc^5$, which determine five divisors of type $(1,1)$ in $\pp^4 \times \pp^4$, symmetric under the action of $G := \zz/2$ that exchanges the two factors, and let $\tilde X \subset \pp^4 \times \pp^4$ be their intersection.  For generic choices, $\tilde X$ is a smooth threefold that does not intersect the diagonal, so $G$ acts freely on $\tilde X$.  Then $X := \tilde X / G$. \bigskip

Next recall that $\Br(X)$ is isomorphic to the torsion subgroup of $H^3(X,\zz)$, because $H^2(\mathcal O_X) = 0$.  We will use the Hochschild--Serre spectral sequence
\[ E_2^{p,q} = H^p(G, H^q(\tilde X, \zz)) \ \Longrightarrow \ H^{p+q}(X,\zz), \]
which is the Leray spectral sequence for the fibration
\[ \xymatrix@C=1em{
\tilde X \ar@{^(->}[rr] && \tilde X \times_G EG \ar[d] \ar@{}[r]|-{\displaystyle \simeq} & X \ \\
&& BG.
} \]
Thus we need to compute the $G$-action on $H^i(\tilde X,\zz)$ for $i = 0, 1, 2, 3$.  In fact we only need a little information about $H^3$, which is lucky because it seems hard to get complete information. \bigskip

By the Lefschetz hyperplane theorem we have
\[ H^i(\tilde X, \zz) = H^i(\pp^4 \times \pp^4, \zz) \qquad \text{ for } i < 3, \]
and similarly with $H_i$.  Thus:
\begin{itemize}
\item $H^0(\tilde X, \zz) = \zz$ with the trivial $G$-action.
\item $H^1(X,\zz) = 0$.
\item $H^2(X,\zz) = \zz^2$ with the $G$-action that exchanges the two factors.
\item $H^3(\tilde X, \zz)$ is torsion free: by the universal coefficient theorem, the torsion in $H^3$ is isomorphic to the torsion in $H_2$.
\end{itemize}
\bigskip

The cohomology of $G$ with coefficients in the trivial module $\zz$ is
\[ H^*(G,\zz) = \zz,\ 0,\ \zz/2,\ 0,\ \zz/2,\ 0,\ \zz/2,\ \dotsc. \]
With coefficients in the regular representation $\zz^2$, also known as $\zz G$, it is
\[ H^*(G,\zz G) = \zz,\ 0,\ 0,\ 0,\ 0,\ 0,\ 0,\ \dotsc. \]
This can be found in many textbooks, for example \cite{df}*{\S17.2}. \bigskip

Now the $E_2$ page of the spectral sequence is

\[ \def \arraystretch {1.2}
\begin{array}{|ccccccccc}
q\uparrow \\
\\
\vdots & \vdots \\
H^3(\tilde X,\zz)^G & ? & ? & ? & ? & ? & \dotsb \\
\zz & 0 & 0 & 0 & 0 & 0 & \dotsb \\
0 & 0 & 0 & 0 & 0 & 0 & \dotsb \\
\zz & 0 & \zz/2 & 0 & \zz/2 & 0 & \dotsb  & & p \to\\
\hline
\end{array} \]
\medskip

\noindent where $H^3(\tilde X, \zz)^G$ denotes the invariant submodule of $H^3(\tilde X, \zz)$.  This is the only term that can contribute to $H^3(X,\zz)$, and it is torsion-free.  On the $E_4$ page there may be a non-zero differential $H^3(\tilde X, \zz)^G \to \zz/2$, but its kernel is still torsion-free.  Thus $H^3(X,\zz)$ is torsion-free, as desired.

\bigskip

\ContactInfo

\end{document}